\documentclass[12pt,a4paper]{article}

\usepackage{bbm}
\usepackage{graphicx}
\usepackage{enumerate}
\usepackage{amssymb, amsmath, amsthm}

\newtheorem{satz}{Theorem}[section]
\newtheorem{lem}[satz]{Lemma}

\newtheorem{cor}[satz]{Corollary}

\newtheorem{anm}[satz]{Remark}

\newtheorem{bsp}[satz]{Example}

\usepackage[T1]{fontenc}
\usepackage{lmodern}

\newcommand{\R}{\ensuremath{{\mathbb R}}}

\newcommand{\Ex}{\mathbb{E}}
\newcommand{\abs}[1]{\left\lvert#1 \right\rvert}
\newcommand{\norm}[1]{\left \lVert#1 \right\rVert}

\begin{document}

\title{An Inversion Formula for Orlicz Norms and Sequences of Random Variables}
\author{S\"oren Christensen \and Joscha Prochno  \and Stiene Riemer }
\date{\today}
\maketitle

\begin{abstract}
Given an Orlicz function $M$, we show which random variables $\xi_i$, $i=1,\ldots,n$ generate the associated Orlicz norm, {\it i.e.}, which random variables yield $\mathbb{E} \max\limits_{1\leq i \leq n}|x_i\xi_i| \sim \norm{(x_i)_{i=1}^n}_M$. As a corollary we obtain a representation for the distribution function in terms of $M$ and $M'$ which can be easily applied to many examples of interest. 
\end{abstract}


\section{Introduction}\label{intro}

Over the years, various new tools and techniques were developed in the study of order statistics, where the $k$-th order statistic of a statistical sample of size $n$ is equal to its $k$-th smallest value. For general literature on order statistics see \cite{key-BC} or \cite{key-DN}. Order statistics are well known tools in non-parametric statistics and find application in wireless networks, compressed sensing and appear naturally in Banach space theory, when one considers sequences of random variables. They are used to study certain parameters associated to convex bodies and also provide a tool to compute the eigenvalue distributions of random matrices. Just to mention a few, see \cite{key-CDV}, \cite{key-GGMP}, \cite{key-G}, \cite{key-MP} or \cite{key-R}. 
Over the last years, there was put tremendous effort in giving estimates for the order of these quantities. Very precise results were obtained in \cite{key-GLSW4} (see also the references therein).\\
In \cite{key-GLSW} the authors studied the important maximal case, {\it i.e.}, the case of the $1$-st order statistic. Given independent identically distributed (iid) random variables $X_1,\ldots,X_n$, the new approach started there, was to give estimates for the quantity
  \begin{equation}\label{EQU Erwartungswert ueber Maximum}
      \mathbb{E}\max\limits_{1\leq i \leq n}|x_iX_i|
  \end{equation}
in terms of Orlicz norms, where the Orlicz norm depends on the distribution of the $X_i$'s (definitions are given below). These estimates turned out to be quite useful, {\it e.g.}, recently, in \cite{key-AP} estimates for support functions of random polytopes were obtained using this approach.    

The case of discrete random variables was first studied in \cite{key-KS1} and \cite{key-KS2}. To be more precise, the authors considered averages over permutations and obtained that for any $x\in \R^n$ and $a_1\geq \ldots \geq a_n>0$
  \begin{equation}\label{EQU discrete}
    \frac{1}{n!} \sum_{\pi} \max_{1\leq i \leq n}\abs{x_ia_{\pi(i)}} \sim \norm{x}_M,
  \end{equation}
where $M$ is defined by $M^*\left(\sum_{i=1}^k a_i\right) = \frac{k}{n}$ for $k=1,\ldots,n$ and extended linearly, where $M^*$ is the dual function of $M$. They also obtained an inverse result, {\it i.e.}, they proved that given an Orlicz function $M$, the choice
  \begin{equation}\label{EQU sequence}
    a_i = n\left( M^{*-1}\left( \frac{i}{n}\right) - M^{*-1}\left( \frac{i-1}{n}\right) \right), i=1,\ldots,n,
  \end{equation}
yields the estimate in (\ref{EQU discrete}).    

Here, we prove the inverse result in the setting of (\ref{EQU Erwartungswert ueber Maximum}), {\it i.e.}, given an Orlicz function $M$, we show how to choose the distribution of the random variables $X_i$, $i=1,\ldots,n$ so that (\ref{EQU Erwartungswert ueber Maximum}) is equivalent to the Orlicz norm of the scalar sequence $(x_i)_{i=1}^n$. The result we obtain (especially Corollary \ref{COR inverse distribution function}) is in most cases easier to apply than (\ref{EQU sequence}), since it is easier to calculate the derivative instead of the inverse of the dual function (see definitions below). Our main result is the following:

\begin{satz}\label{Hauptsatz} 
    Let $M$ be an Orlicz function with right-hand derivative $M^{'}$ such that $M^{'}(0)=0$. We define 
    \[Q(\cdot)=\int y \delta_{\frac{1}{y}}(\cdot) dM^{'}(y).\]
   Whenever 
    \begin{equation}\label{eq:finite}
    \int y dM^{'}(y)<\infty,
    \end{equation}
    we get for all $X_1,\ldots,X_n$ independent, distributed according to $Q(\cdot)/Q([0,\infty))$ and all $x=(x_i)_{i=1}^n\in\R^n$
    $$
       c_1\norm{x}_M\leq \mathbb{E}\max\limits_{1\leq i \leq n}|x_iX_i|\leq c_2\norm{x}_M,
    $$
    where $c_1,c_2$ are positive absolute constants.
\end{satz}

It was shown in \cite{key-GLSW} that for iid $p$-stable random variables  $X_1,\ldots,X_n$, $1<p<2$, and for all $x=(x_i)_{i=1}^n\in\R^n$
$$
   c_1\norm{x}_p\leq E\max\limits_{1\leq i \leq n}|x_iX_i|\leq c_2\norm{x}_p,
$$
{\it i.e.}, a $p$-stable distribution generates the $\ell_p$ norm. Here, $p$-stable means that the characteristic function of $X_i$ is given by $\psi(u)=\exp(-|u|^p)$.
Furthermore, in \cite{key-GLSW} and \cite{key-GLSW4} the authors proved that for independent, identically standard gaussian random variables $X_1,\ldots,X_n$, and for all $x=(x_i)_{i=1}^n\in\R^n$
$$
   c_1\norm{x}_M\leq \mathbb{E}\max\limits_{1\leq i \leq n}|x_iX_i|\leq c_2\norm{x}_M,
$$
where
  $$
      M(s) = \sqrt{\frac{2}{\pi}} \int_0^s e^{-\frac{1}{2t^2}} dt.
  $$ 
However, this obviously does not give the $\ell_2$-norm. In fact, the obtained Orlicz norm is far from that, because of the gaussian tails. At first this seems astonishing, since we deal with the case of a 2-stable distribution. For $p=2$, one could expect to get the $\ell_2$-norm, since a sequence of independent gaussian random variables spans a subspace of $L_p$ isometric isomorphic to $\ell_2$. On the other hand, from a probabilistic point of view this seems somehow natural, because a $2$-stable distribution is very different from $p$-stable distributions for $1<p<2$, {\it e.g.}, existence of moments.

Therefore, the question arose which distribution generates the $\ell_2$-norm. The second and third named author gave an answer in \cite{key-PR}. However, this answer is limited to the special situation and it seems that it cannot be generalized. In fact, our main result here immediately yields that $\log-\gamma_{1,p}$, $1<p<\infty$ distributed random variables generate a truncated $\ell_p$-norm and especially for $p=2$, we obtain the $\ell_2$-norm.

\section{Preliminaries}

We recall that a function $M:[0,\infty) \to [0,\infty)$ is called an Orlicz function, if it is convex, $M(0)=0$ and $M(t)>0$ for all $t>0$. 
If there exists a $t_0>0$ such that $M(t)=0$ for $t\leq t_0$, we call it a degenerated Orlicz function.  
The dual function $M^*$ of an Orlicz function $M$ is given by the Legendre transform
  $$
    M^*(x) = \sup_{t\in[0,\infty)}(xt-M(t)).
  $$
Again, $M^*$ is an Orlicz function and $M^{**}=M$. Hence, $M^*$ uniquely determines $M$. In case $M(t)=t^p$, $1\leq p<\infty$, we just have $\norm{\cdot}_M=\norm{\cdot}_p$. For all $x=(x_i)_{i=1}^n\in\R^n$, we define the Orlicz norm $\norm{\cdot}_M$ by 
$$
   \norm{x}_M=\inf\left\{t>0 : \sum\limits_{i=1}\limits^nM\left(\frac{|x_i|}{t}\right)\leq 1\right\}
$$
and define the Orlicz space $\ell_M^n$ to be $\R^n$ equipped with $\norm{\cdot}_M$.
For more details and a thorough introduction to Orlicz spaces we refer the reader to \cite{key-KR} or \cite{key-RR}.

We will use the notation $a\sim b$ to express that there exist two positive absolute constants $c_1, c_2$ such that $c_1a\leq b\leq c_2 a$ and use $a\sim_{p} b$ in case the constants depend on some parameter $p>0$.\\

Since the Lebesgue-Stieltjes integral allows integration by parts, we obtain the following result:

  \begin{lem}\label{Choquet-type Representation}
    Let $M$ be an Orlicz function (possibly degenerated) with right-hand derivative $M'$. Then, for all $x\geq 0$,
      $$
        M(x) = \int_{[0,\infty)} (x-y)^+ dM'(y).
      $$
  \end{lem}
  \begin{proof}
We use integration by parts. For all $x \geq 0$ 
      \begin{eqnarray*}
        && \int_{[0,\infty)}(x-y)^+ dM'(y) \\
        & = & xM'(0) + \int_{(0,x]} (x-y)^+dM'(y) \\
        & = & xM'(0) + (x-x)^+M'(x)-(x-0)^+M'(0) - \int_{(0,x]} M'(y)d(x-y)^+ \\
        & = & \int_{(0,x]} M'(y)dy ~=~ M(x).
      \end{eqnarray*}
  \end{proof}
This easy observation, however, forms the starting point for our inversion formula. 

\begin{anm}\label{rem:choquet}
This integral representation is of Choquet-type, since the functions $x\mapsto(x-y)^+$ are (up to normalization) the extreme points of the convex cone of Orlicz functions.
\end{anm}

In \cite{key-GLSW4} the authors proved the following theorem:

\begin{satz}\label{THM Schuett u Werner}
Let $X_1,\ldots,X_n$ be iid random variables and $\mathbb E \abs{X_1} < \infty$. For all $s\in\R_{\geq 0}$ let
$$
  M(s)=\int\limits_0\limits^s\int\limits_{\frac{1}{t}\leq |X_1|}|X_1|dPdt.
$$ 
Then, for all $x=(x_i)_{i=1}^n\in\R^n$,
$$
   c_1\norm{x}_M\leq E\max_{1\leq i \leq n}|x_iX_i|\leq c_2\norm{x}_M,
$$
where $c_1,c_2$ are positive absolute constants.
\end{satz}

The theorem says that given a certain distribution, one can generate a corresponding Orlicz norm, {\it e.g.}, $p$-stable random variables generate the $\ell_p$-norm.
We would like to point out that $M$ can also be written in the following way
  \begin{equation}\label{EQU representation orlicz function}
    M(s) = \int_0^s \left(\tfrac{1}{t} \mathbb P (\abs{X}\geq \tfrac{1}{t}) + \int_{\frac{1}{t}}^{\infty} \mathbb P (\abs{X}\geq u) du \right) dt.
  \end{equation}

\section{The Inversion Formula}

Before we give the proof of our main theorem, we want to make some remarks to get a better understanding. First of all, note that condition \eqref{eq:finite} is no restriction, because to define an Orlicz norm $\norm{\cdot}_M$, 
it is enough to define an Orlicz function $M$ on the interval $[0,T]$, where $M(T)=1$, and then extend it linearly.

Furthermore, whenever $M^{'}(0)=0$ we see that $Q(\cdot) = \int y \delta_{\frac{1}{y}}(\cdot) dM^{'}(y)$ obviously defines a measure on $\R$, and since \eqref{eq:finite} is no restriction, $Q(\cdot)/Q([0,\infty))$ is a probability measure. The distribution function of $Q(\cdot)$ is given by
     $$
         F(x) = Q((-\infty,x]) = \int y \delta_{\frac{1}{y}}((-\infty,x]) dM^{'}(y) = \int\limits_{1/x}\limits^{\infty} y dM'(y),
     $$
since $\delta_{\frac{1}{y}}((-\infty,x])=1$ whenever $\tfrac{1}{y}\in(-\infty,x]$, i.e., if $y \geq \tfrac{1}{x}$. 
In the same way we obtain the tail distribution function $\bar F= 1-F$, i.e., for all $x>0$,
    \begin{equation}\label{EQU inverse distribution of Q}
       \bar F (x) = \int_0^{\frac{1}{x}} y dM'(y),
    \end{equation}
where we restrict ourselves to the non-negative real line, since we defined our Orlicz functions on $[0,\infty)$.     

We would also like to point out that in the setting of Theorem \ref{Hauptsatz}, for all Q-integrable function $h$
      \begin{equation} \label{EQU Zerlegung des Masses Q}
         \int h(x)dQ(x) = \int \int h(x) dQ_y(x) dM'(y),
      \end{equation}
where $Q_y(\cdot) = y\delta_{\frac{1}{y}}(\cdot)$, $y>0$.      
We will now prove the main theorem.

\begin{proof}[Proof of Theorem \ref{Hauptsatz}]
  For some $y>0$, let $Q_y(\cdot) = y\delta_{\frac{1}{y}}(\cdot)$. Then, we have for all $s\geq 0$,
    \begin{equation}\label{EQU Integral liefert Positivteil}
       \int_0^s \int_{\frac{1}{t} \leq x} x dQ_y(x) dt  =  y \int_0^s \int_{\frac{1}{t} \leq x} x d\delta_{\frac{1}{y}}(x) dt. \\
     \end{equation}
  Now, whenever $\frac{1}{t} \leq \frac{1}{y}$, we obtain
    $$
       \int_{\frac{1}{t} \leq x} x d\delta_{\frac{1}{y}}(x) = \frac{1}{y},
    $$
  and $0$ otherwise. So we have
\begin{equation}\label{eq:extreme points}
       \int_0^s \int_{\frac{1}{t} \leq x} x dQ_y(x) dt = \int_0^s \mathbbm{1}_{\{\frac{1}{t}\leq \frac{1}{y}\}} dt
       = (s-y)^+.
\end{equation}
  Therefore, using (\ref{EQU Zerlegung des Masses Q}), we obtain
    $$
       \int_0^s \int_{\frac{1}{t}\leq x} x dQ(x)dt = \int \int_0^s \int_{\frac{1}{t}\leq x} x dQ_y(x)dt dM'(y). 
    $$
  Using equation (\ref{EQU Integral liefert Positivteil}), we get
    $$
       \int \int_0^s \int_{\frac{1}{t}\leq x} x dQ_y(x)dt dM'(y) = \int (s-y)^+dM'(y).
    $$
  Lemma \ref{Choquet-type Representation} yields
    $$
        \int_0^s \int_{\frac{1}{t}\leq x} x dQ(x)dt = M(s).
    $$
   Now, note that $\int ydQ(y)<\infty$. Therefore, by Theorem \ref{THM Schuett u Werner}, there exist positive absolute constants $c_1,c_2$ such that for all
  $x=(x_i)_{i=1}^n$
    $$
       c_1\norm{x}_M\leq E\max\limits_{1\leq i \leq n}|x_iX_i|\leq c_2\norm{x}_M.
    $$ 
\end{proof}

\begin{anm}
Keeping Remark \ref{rem:choquet} in mind, equation \eqref{eq:extreme points} shows that the extreme points of the convex cone of Orlicz functions correspond to the extreme points of the convex cone of probability measures - the Dirac measures. This is the key idea to obtain the inversion formula.
\end{anm}

As a corollary we obtain the following useful representation of the tail distribution function $\bar F = 1- F$.

\begin{cor}\label{COR inverse distribution function}
Let $M$ be a twice  continuously differentiable Orlicz function satisfying \eqref{eq:finite} and $M'(0)=0$. Then, for any random variable $X$ distributed according to $Q(\cdot)/Q([0,\infty))=\int y \delta_{\frac{1}{y}}(\cdot) dM^{'}(y)/Q([0,\infty))$, and for all $x>0$,
$$
    \bar F(x) = \bar F_X (x) = \int_0^{1/x} y M^{''}(y)dy =\frac{1}{x}M'\left(\frac{1}{x}\right) - M\left(\frac{1}{x}\right) .
$$
In particular, for all $x>0$,
$$
    \bar F(x)=\int_{x}^{\infty} \frac{1}{y^3} M''\left(\frac{1}{y}\right)dy.
$$  
\end{cor}
\begin{proof}
  By Theorem \ref{Hauptsatz}, we can represent $M$ in the following way
    $$
        M(s) = \int\limits_0\limits^s\int\limits_{\frac{1}{t}\leq |x|}|x|dQ(x)dt,
    $$
  where $Q(\cdot) = \int y \delta_{\frac{1}{y}}(\cdot) dM'(y)$.  Equation (\ref{EQU inverse distribution of Q}) gives
    $$
        \bar F(x) = \int_0^{1/x} y dM'(y).
    $$
  Since $M$ is twice differentiable
    $$
        \int_0^{1/x} y dM'(y) = \int_0^{1/x} y M''(y) dy.
    $$
  Integration by parts yields 
    $$
       \int_0^{1/x} y M''(y) dy = \frac{1}{x}M'\left(\frac{1}{x}\right) - M\left(\frac{1}{x}\right).
    $$
  In particular, by a change of variables,
    $$
        \bar F(x)=\int_{x}^{\infty} \frac{1}{y^3} M''\left(\frac{1}{y}\right)dy.
    $$     
\end{proof}

Thus, in the case of a twice continuously differentiable Orlicz function, we have a continuous density function.  Especially, we obtain the distribution function
  $$
      F(x) = 1- \left\{\frac{1}{x}M'\left(\frac{1}{x}\right)-M\left(\frac{1}{x}\right)\right\}.
  $$
Furthermore, Corollary \ref{COR inverse distribution function} says that for any $x>0$
  $$
      \mathbb P (X \geq x) = \frac{1}{x}M'\left(\frac{1}{x}\right) - M\left(\frac{1}{x}\right),
  $$ 
where $X$ is distributed according to $Q(\cdot)/Q([0,\infty))$.

Now we consider two examples to show how easily our result can be applied to obtain the distribution of the random variables so that they generate the given Orlicz norm. We start with the case of $\ell_p$-norms.

\begin{bsp}  
  Let $1<p<\infty$ and $M: [0,\infty)\to[0,\infty)$, $t\mapsto t^p$. Then, for all $x=(x_i)_{i=1}^n\in\R^n$, we get
	$$
	     \Ex\max\limits_{1\leq i \leq n}|x_iX_i|\sim_p \norm{x}_{p},
          $$
 in the case that $X_i$, $i=1,\ldots,n$ are iid random variables with   
    $$
       \mathbb P ( X_i \geq x )=px^{-p}
    $$
  for all $x>0$.   
\end{bsp}
 \begin{proof}         
	By Corollary \ref{COR inverse distribution function} we know that the tail distribution function of $X_i$, $i=1,\ldots,n$, is given by
	$$\bar{F}(x)=\int\limits_0\limits^{\frac{1}{x}}yM^{''}(y)dy,$$
	where $M(y)=y^p$. Thus, we get
	\begin{eqnarray*}
		\bar{F}(x)=\int\limits_0\limits^{\frac{1}{x}}yp(p-1)y^{p-2}dy=px^{-p}.
	\end{eqnarray*}
	Therefore, the distribution function $F$ has to be of the form $F(x)=1-px^{-p}$,
	which is up to an absolute constant the same as the distribution function of a $\log\gamma(p,1)$-distributed random variable. Using Theorem \ref{THM Schuett u Werner} it is easy to show that those 
	random variables generate the $\ell_p$-norm (see \cite{key-PR}).\\
\end{proof}

We would also like to emphasize that it is important to know the distribution which yields the $\ell_2$-norm, since this is a main ingredient to obtain embeddings into $L_1$ (see \cite{key-P} and \cite{key-S1}).

The following example is in connection with gaussian random variables.

\begin{bsp}
  Let $M: [0,\infty)\to[0,\infty)$ be given by
    \begin{equation} \label{EQU orlicz function for gaussians}
        M(s) = \sqrt{\frac{2}{\pi}} \int_0^s e^{-\frac{1}{2t^2}} dt.
    \end{equation}
  Then, for all $x=(x_i)_{i=1}^n\in\R^n$, we get
  $$
       \Ex\max\limits_{1\leq i \leq n}|x_iX_i|\sim \norm{x}_{M},
  $$  
 in the case that $X_i$, $i=1,\ldots,n$ are iid random variables with
  $$
      \mathbb P( X_i \geq x) =  \frac{1}{\sqrt{2\pi}} \int_x^{\infty} e^{-\frac{y^2}{2}} dy
  $$ 
 for all $x>0$.  
\end{bsp}
\begin{proof}
  Again, to obtain the distribution of the random variables $X_i$, $i=1,\ldots,n$ which generate the norm given by (\ref{EQU orlicz function for gaussians}), we use Corollary \ref{COR inverse distribution function} and get
  $$
      \bar{F}(x)=\int_{x}^{\infty} \frac{1}{y^3} M''\left(\frac{1}{y}\right)dy.
  $$
So an easy computation immediately yields that for any $x>0$
  $$
      \mathbb P (X_i \geq x) = \frac{1}{\sqrt{2\pi}} \int_x^{\infty} e^{\frac{y^2}{2}}dy.
  $$
Theorem \ref{THM Schuett u Werner} with the representation of $M$ given in (\ref{EQU representation orlicz function}) shows that these random variables generate the Orlicz norm $\norm{\cdot}_M$.  	
\end{proof}

The last example shows that in order to obtain the Orlicz norm given by the Orlicz function defined in (\ref{EQU orlicz function for gaussians}), we can take iid standard gaussian random variables.

~\\
~\\
{\bf S\"oren Christensen}\\
Mathematisches Seminar\\
Christian-Albrechts-Universit\"at zu Kiel\\
Ludewig Meyn Str. 4\\
24098 Kiel, Germany\\
{\em christensen@math.uni-kiel.de}\\
~\\
~\\
{\bf Joscha Prochno}\\
Department of Mathematical and Statistical Sciences\\
University of Alberta\\
524 Central Academic Building\\
Edmonton, Alberta\\
Canada T6G 2G1\\
{\em prochno@ualberta.ca}\\
~\\
~\\
{\bf Stiene Riemer}\\
Mathematisches Seminar\\
Christian-Albrechts-Universit\"at zu Kiel\\
Ludewig Meyn Str. 4\\
24098 Kiel, Germany\\
{\em riemer@math.uni-kiel.de}\\

\end{document}